\documentclass[11pt]{amsart}
\usepackage{amssymb}
\usepackage{amsmath}
\usepackage{MnSymbol}
\usepackage{amsthm}
\usepackage{graphicx}
\usepackage{xcolor}
\usepackage{marvosym}

\newcommand     {\comment}[1]   {}

\newtheorem{lemma}{Lemma}[section]
\newtheorem{corollary}[lemma]{Corollary}
\newtheorem{remark}[lemma]{Remark}

\begin{document}

\vspace*{-17ex}

\title{A Chiang-type Lagrangian in $\mathbb{CP}^2$}

\author{Ana Cannas da Silva}

%\thanks{Com o apoio da
%Funda\c{c}\~{a}o para a Ci\^{e}ncia e a Tecnologia (FCT/Portugal).}

\address{Department of Mathematics, ETH Zurich, R\"amistrasse 101,
8092 Z\"urich, Switzerland}
\email{acannas@math.ethz.ch}

\date{July 2015}

\maketitle

\begin{abstract}
We analyse a simple Chiang-type lagrangian in $\mathbb{CP}^2$
%the complex projective plane
which is topologically an $\mathbb{RP}^2$
%a real projective plane
but exhibits a distinguishing behaviour under
reduction by one of the toric circle actions, in particular
it is a {\em one-to-one transverse lifting}
of a great circle in $\mathbb{CP}^1$.
\end{abstract}

\vspace*{3ex}

%%%%%%%%%%%%%%%%%%%%%%%%%%%%%%%%%%%%%%%%%%%%%%%%%%%%%%%%%%%%%%%

\section[]{Lagrangians in a symplectic reduction scenario}

%%%%%%%%%%%%%%%%%%%%%%%%%%%%%%%%%%%%%%%%%%%%%%%%%%%%%%%%%%%%%%%

Let $(M,\omega)$ be a $2n$-dimensional symplectic manifold
equipped with a hamiltonian action of a
$k$-dimensional real torus $T^k$ where $k < n$ and corresponding moment map
\[
   \mu : M \to \mathbb{R}^k \ .
\]

We assume that $a \in \mathbb{R}^k$ is a regular value of $\mu$ such that
the reduced space
\[
   M_{red} = \mu^{-1}(a) / T^k
\]
is a manifold.
The latter comes with the so-called {\em reduced} symplectic form
$\omega_{red}$ satisfying the equation $\pi^* \omega_{red} = \imath^* \omega$,
where $\pi: \mu^{-1}(a) \to M_{red}$ is the point-orbit projection and
$\imath : \mu^{-1}(a) \hookrightarrow M$ is the set inclusion:
\[
\begin{array}{rcl}
   \mu^{-1}(a) & \stackrel{\imath}{\hookrightarrow} & M \\
   \downarrow \pi \\
   M_{red}
\end{array}
\]

Let $\ell_1 \subset M_{red}$ be a compact lagrangian submanifold
of the reduced space.
Then its preimage in $M$,
\[
    L_1 := \imath( \pi^{-1} (\ell_1)) \ ,
\]
is always a compact lagrangian submanifold of $(M, \omega)$,
which happens to lie entirely in the level set $\mu^{-1}(a)$.

Motivated by a question of Katrin Wehrheim's related to her
joint work with Chris Woodward~\cite{ww_quilted},
we are looking for a reverse picture where -- starting from a lagrangian
in $(M, \omega)$ -- we obtain a lagrangian in $(M_{red},\omega_{red})$.
For that purpose suppose now that we have a compact lagrangian
submanifold $L_2 \subset M$
such that
\begin{itemize}
\item
$L_2 \pitchfork \mu^{-1}(a)$, that is, $L_2$ and $\mu^{-1}(a)$ {\em intersect
transversally} and
\item
$L_2 \cap \mu^{-1}(a) \stackrel{\pi}{\hookrightarrow} M_{red}$,
that is, $L_2$ intersects
%add pi
each $T^n$-orbit in $\mu^{-1}(a)$ at most once,
\end{itemize}
so the intersection submanifold $L_2 \cap \mu^{-1}(a)$ injects
into the reduced space $M_{red}$ via the point-orbit projection $\pi$.
In this case we obtain an embedded lagrangian submanifold in the reduced
space $(M_{red},\omega_{red})$, namely
\[
   \ell_2 := \pi (L_2 \cap \mu^{-1}(a)) \ ,
\]
and we call $L_2$ a {\em one-to-one transverse lifting} of $\ell_2$.

In this note we will concentrate on the case where the symplectic manifold
is the complex projective plane, $M = \mathbb{CP}^2$, with a
scaled Fubini-Study structure so that the total volume is $\frac{\pi^2}{2}$.
We regard the circle action
where a circle element $e^{i\theta} \in S^1$ ($0 \leq \theta < 2\pi$) acts by
\[
   [z_0:z_1:z_2] \longmapsto [z_0:z_1:e^{i\theta}z_2]
\]
with moment map
\[
    \mu_2 : \mathbb{CP}^2 \to \mathbb{R} \ , \quad
    \mu_2 [z_0:z_1:z_2] = -\frac 12 \frac{|z_2|^2}{|z_0|^2+|z_1|^2+|z_2|^2}
\]
and the chosen level is $a = -\frac 16$.\footnote{All other levels
in the range $-\frac 14 < a < 0$ allow the {\em one-to-one transverse
lifting property}, but this level (that of the so-called
{\em Clifford torus} $\{[1:e^{i\theta_1}:e^{i\theta_2}]\}$)
ensures the monotonicity of $\ell_2$; see Remark~\ref{remark_monotonicity}.}
With these choices we have $n=2$, $k=1$ and the reduced space
$(M_{red},\omega_{red})$ is a
complex projective line $\mathbb{CP}^1$ with a scaled Fubini-Study form
so that the total area is $\frac{2\pi}{3}$.

Pictorially in terms of toric moment polytopes we have for
the original space $\mathbb{CP}^2$:

%\par\bigskip

\begin{center}
\includegraphics[scale=0.7]{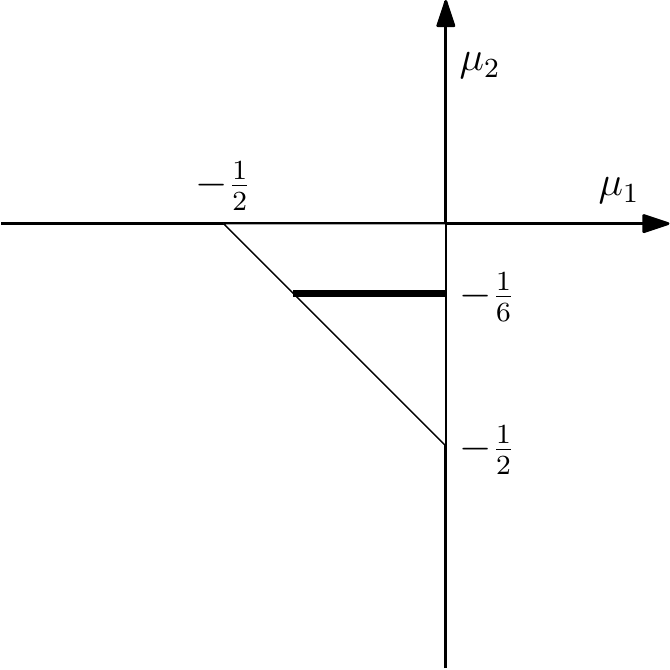}
\end{center}

%\par\bigskip

\noindent
and for the reduced space $\mathbb{CP}^1$:

\begin{center}
\includegraphics[scale=0.7]{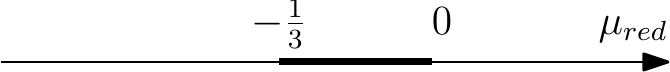}
\end{center}

%\begin{figure}[h]
%      \centering
%      \includegraphics{Fig1.pdf}
%      \caption{Faixas esf\'ericas e cil\'{\i}ndricas com a mesma \'area.}
%      \label{fig:1}
%  \end{figure}

The reduction is carried out with respect to the moment map $\mu_2$.
The other standard hamiltonian circle action with moment map
$\mu_1 [z_0:z_1:z_2] = -\frac 12 \frac{|z_1|^2}{|z_0|^2+|z_1|^2+|z_2|^2}$
descends to the reduced space, where we denote the induced
moment map $\mu_{red} : \mathbb{CP}^1 \to \mathbb{R}$.

Section~\ref{section_l2} exhibits a non-standard lagrangian embedding
of $\mathbb{RP}^2$ into $\mathbb{CP}^2$,
which in Section~\ref{section_properties} is shown to be a
{\em one-to-one transverse lifting} of a great circle in $\mathbb{CP}^1$.

%%%%%%%%%%%%%%%%%%%%%%%%%%%%%%%%%%%%%%%%%%%%%%%%%%%%%%%%%%%%%%%

\section[]{A non-standard lagrangian $\mathbb{RP}^2$
in $\mathbb{CP}^2$ from representation theory}
\label{section_l2}

The example at stake is a lagrangian submanifold $\mathcal{L}_2$ of
$\mathbb{CP}^2$ which arises as an orbit for an
$SU(2)$-action, similar to a lagrangian studied by
River Chiang in~\cite{chiangs_lagrangians}.
We first define $\mathcal{L}_2$ explicitly as a subset of $\mathbb{CP}^2$:
\[
   \mathcal{L}_2 := \left\{ \left[ \bar{\alpha}^2 + \bar{\beta}^2 :
   \sqrt{2}(\bar{\alpha} \beta - \alpha \bar{\beta}) :
   \alpha^2 + \beta^2 \right] \; \mid \; \alpha, \beta \in \mathbb{C} \ ,
   |\alpha|^2+|\beta|^2=1 \right\} \ .
\]

\begin{lemma}
The set $\mathcal{L}_2$ is a submanifold of
$\mathbb{CP}^2$ diffeomorphic to $\mathbb{RP}^2$.
\end{lemma}

\begin{proof}
We view $\mathbb{CP}^2$ as the space of homogeneous complex polynomials
of degree 2 in two variables $x$ and $y$ up to scaling,
\[
   [ z_0 : z_1 : z_2 ] \quad \longleftrightarrow \quad
   p_{[ z_0 : z_1 : z_2 ]} (x,y) = z_0 y^2 + z_1 \sqrt{2} xy + z_2 x^2 \ ,
\]
and let the group $SU(2)$ act as follows.
An element
\[
   A = \begin{pmatrix}
     \alpha & \beta \\
     -\bar{\beta} & \bar{\alpha}
     \end{pmatrix} \in SU(2) \quad \mbox{ where } \quad
   \alpha, \beta \in \mathbb{C} \mbox{ and } |\alpha|^2+|\beta|^2=1 \ ,
\]
acts by the change-of-variables diffeomorphism
\[
   p(x,y) \quad \longmapsto \quad p \left( (x,y) A \right)
   = p \left( \alpha x - \bar{\beta} y ,
   \beta x + \bar{\alpha} y \right) \ .
\]
Then the set $\mathcal{L}_2$ is simply the $SU(2)$-orbit of the polynomial
$p_{[1:0:1]}(x,y) = y^2+x^2$
or, equivalently, of the point $[1:0:1]$.
%similarly to the situation in~\cite{chiangs_lagrangians}.

So $\mathcal{L}_2$ is an $SU(2)$-homogeneous space.
Since the stabilizer of $[1:0:1] \in \mathcal{L}_2$ is the
subgroup
\[
   \mathcal{S} := \left\{ \begin{pmatrix}
   \cos \theta & \sin \theta \\
   -\sin \theta & \cos \theta \end{pmatrix} \ ,
   \begin{pmatrix}
   i \cos \theta & i \sin \theta \\
   i \sin \theta & -i \cos \theta \end{pmatrix}
   \mid \theta \in [0,2\pi) \right\}
\]
where the component of the identity is a circle subgroup of $SU(2)$
(hence conjugate to the Hopf subgroup),
%the semidirect product $\mathcal{S}^1 \rtimes \mathcal{Z}_2$ where
%\[
%   \mathcal{S}^1 := \left\{ .............. \right\}
%\]
%is a Hopf-like circle subgroup of $SU(2)$ and
%\[
%   \mathbb{Z}_2 :=
%\]
it follows that $\mathcal{L}_2$ is diffeomorphic to
\[
   SU(2) \; / \; \mathcal{S} \simeq \mathbb{RP}^2 \ .
\]
\end{proof}

There is a convenient real description of $\mathcal{L}_2$ by
defining
\[
   X := \Re (\alpha^2 + \beta^2) \ , \quad
   Y := \Im (\alpha^2 + \beta^2) \quad \mbox{ and } \quad
   Z := 2 \Im (\bar{\alpha} \beta) \ .
\]
In these terms we have
\[
   \mathcal{L}_2 = \left\{ \left[ X-iY : \sqrt{2}iZ : X+iY \right] \; \mid \;
   X, Y, Z \in \mathbb{R} \ , X^2+Y^2+Z^2 = 1 \right\} \ ,
\]
which also shows that
$\mathcal{L}_2 \simeq S^2 / \pm 1 \simeq \mathbb{RP}^2 \ .$

\begin{lemma}
The submanifold $\mathcal{L}_2$ is lagrangian.
\end{lemma}

\begin{proof}
The fact that $\mathcal{L}_2$ is isotropic can be checked either
directly (computing the vector fields generated by the $SU(2)$ action)
or, as done here, analysing a hamiltonian action in the
context of representation theory.

The standard hermitian metric on the vector space $V := \mathbb{C}^2$,
namely
\[
   h(z,w) = \bar{z}^T w =
   \underbrace{\Re (\bar{z}^T w)}_{\begin{array}{c}\text{\tiny{euclidean}}\\ \text{\tiny{inner prod.}}\end{array}}
   + i\underbrace{\Im (\bar{z}^T w)}_{\begin{array}{c}\text{\tiny{standard}}\\ \text{\tiny{sympl.\ str.}}\end{array}} \ ,
\]
induces an hermitian metric on the symmetric power $Sym^2 (V^*)$
for which
\[
   y^2 \ , \quad \sqrt{2} xy \ , \quad x^2
\]
is a unitary basis; here $x$ and $y$ (the linear maps
extracting the first and second components of a vector in $\mathbb{C}^2$)
form the standard unitary basis of $V^*$.
We denote $u_0, u_1, u_2$ the corresponding coordinates in $Sym^2 (V^*)$
and write an element of $Sym^2 (V^*)$ as
\[
   p(x,y) = u_0 y^2 + u_1 \sqrt{2} xy + u_2 x^2 \ .
\]
W.r.t.\ these coordinates the symplectic structure on $Sym^2 (V^*)$
is the standard structure:
\[
   \Omega_0 = \textstyle{\frac{i}{2}} \left( du_0 d\bar{u}_0 +
   du_1 d\bar{u}_1 + du_2 d\bar{u}_2 \right) \ .
\]

The diagonal circle action on $Sym^2 (V^*)$ is hamiltonian
with moment map linearly proportional to $|u_0|^2 + |u_1|^2 + |u_2|^2$
%***************** determine proportionality constant
and by symplectic reduction at the level of the unit sphere
we obtain the space of homogeneous complex polynomials of degree 2
in two variables $x$ and $y$ up to scaling,
$\mathbb{P} \left( Sym^2 (V^*) \right) \simeq S^5/S^1 \simeq \mathbb{CP}^2$,
with the standard Fubini-Study structure
($S^1$ acts on $S^5$ by the Hopf action).

Now the action of $SU(2)$ on $Sym^2 (V^*)$ by change-of-variables as above
is a {\em unitary} representation (this is the 3-dimensional
irreducible representation of $SU(2)$), hence hamiltonian with
moment map\footnote{In general a symplectic representation
$\rho : G \to Sympl(\mathbb{C}^N, \Omega_0)$ may be viewed as a hamiltonian
action with moment map $\tilde{\mu} : \mathbb{C}^N \to \mathfrak{g}^*$,
$\tilde{\mu}(z)(X^{\mathfrak{g}}) = \frac{i}{4}(z^*Xz-z^*X^*z)$ where
$z \in \mathbb{C}^N$, $X^{\mathfrak{g}} \in \mathfrak{g}$ and $X$ the
matrix representing the corresponding element in the Lie algebra
of $Sympl(\mathbb{C}^N, \Omega_0)$.
In the special case where the representation is {\em unitary}, the
formula for the moment map reduces to
$\tilde{\mu}(z)(X^{\mathfrak{g}}) = \frac{i}{2}z^*Xz$.
In the present case one determines the matrix $X$ for each of the Lie algebra
basis elements $X_a, X_b, X_c \in su(2)$.}
\[
   \tilde{\mu} : Sym^2 (V^*) \longrightarrow su(2)^* \simeq \mathbb{R}^3
\]
defined as follows w.r.t.\ the real basis
\[
   X_a := \begin{pmatrix} i & 0 \\ 0 & -i \end{pmatrix} \ , \quad
   X_b := \begin{pmatrix} 0 & i \\ i & 0 \end{pmatrix} \ , \quad
   X_c := \begin{pmatrix} 0 & 1 \\ -1 & 0 \end{pmatrix} \ ,
\]
of the Lie algebra $su(2)$ and the unitary basis above of $Sym^2 (V^*)$:
\[
   \tilde{\mu} (u_0,u_1,u_2) =
   \begin{pmatrix} |u_0|^2 - |u_2|^2 \\
   -\sqrt 2 \Re (u_0\bar{u}_1+u_1\bar{u}_2) \\
   -\sqrt 2 \Im (u_0\bar{u}_1+u_1\bar{u}_2) \end{pmatrix} \ .
\]
This $SU(2)$ action commutes with the diagonal $S^1$ action,
hence descends to a hamiltonian action of $SU(2)$ on $\mathbb{CP}^2$
with moment map
\[
   \tilde{\tilde{\mu}} : \mathbb{CP}^2 \longrightarrow su(2)^* \simeq \mathbb{R}^3 \ ,
   \quad \tilde{\tilde{\mu}} [z_0:z_1:z_2] =
   \begin{pmatrix} |z_0|^2 - |z_2|^2 \\
   -\sqrt 2 \Re (z_0\bar{z}_1+z_1\bar{z}_2) \\
   -\sqrt 2 \Im (z_0\bar{z}_1+z_1\bar{z}_2) \end{pmatrix} \ ,
\]
where now $z_0,z_1,z_2$ are homogeneous coordinates satisfying
$|z_0|^2 + |z_1|^2 + |z_2|^2 =1$.

The submanifold $\mathcal{L}_2$ is isotropic because it is an
$SU(2)$ orbit and sits in the zero-level of the moment map
$\tilde{\tilde{\mu}}$
-- indeed we have $|z_0|^2 = |z_2|^2$ and $z_0\bar{z}_1+z_1\bar{z}_2=0$
when $z_0 = \bar{z}_2$ and $z_1$ is imaginary.
It follows that $\mathcal{L}_2$ is lagrangian because it is half-dimensional.
\end{proof}

%%%%%%%%%%%%%%%%%%%%%%%%%%%%%%%%%%%%%%%%%%%%%%%%%%%%%%%%%%%%%%%

\section[]{Properties of the lagrangian $\mathcal{L}_2$}
\label{section_properties}

\begin{lemma}
\label{lemma_tranversality}
The submanifold $\mathcal{L}_2$ intersects transversally
the moment map level set $\mu_2^{-1}(a)$.
\end{lemma}

\begin{proof}
We claim that in the interior of the moment polytope
the restriction of $d\mu_2$ to $\mathcal{L}_2$ never vanishes.
This implies the lemma because the level set $\mu_2^{-1}(a)$
is a codimension one submanifold whose tangent bundle
is the kernel of $d\mu_2$.

Using the real description of $\mathcal{L}_2$,
\[
   \mathcal{L}_2 = \left\{ \left[ X-iY : \sqrt{2}iZ : X+iY \right] \; \mid \;
   X, Y, Z \in \mathbb{R} \ , X^2+Y^2+Z^2 = 1 \right\} \ ,
\]
we parametrize it (in the interior of the moment polytope)
via real coordinates $X$ and $Y$, choosing
\[
   Z = \sqrt{1-X^2-Y^2} \ , \qquad X^2+Y^2 < 1 \ ,
\]
so that that part of $\mathcal{L}_2$ is the set of points
\[
   \left[ \frac{X-iY}{\sqrt{2}i\sqrt{1-X^2-Y^2}} : 1 :
   \frac{X+iY}{\sqrt{2}i\sqrt{1-X^2-Y^2}} \right] \ .
\]
We evaluate $\mu_2$ in points of the above form:
\[
   \mu_2 \textstyle{\left[ \frac{X-iY}{\sqrt{2}i\sqrt{1-X^2-Y^2}} : 1 :
   \frac{X+iY}{\sqrt{2}i\sqrt{1-X^2-Y^2}} \right]} =
   -\frac 12 \frac{|X+iY|^2}{|X-iY|^2+2(1-X^2-Y^2)+|X+iY|^2} =
   -\frac 14 (X^2+Y^2) \ .
\]
This shows that the differential of $\mu_2$ only vanishes when
$\mu_2$ itself vanishes, which never happens in the interior of the
moment polytope.
\end{proof}

It is interesting also to discuss the image of $\mathcal{L}_2$
by the full toric moment map $(\mu_1,\mu_2)$.
This image satisfies the equation $\mu_1 + 2\mu_2 = - \frac 12$ because
\[
   -\frac 12 \frac{|z_1|^2+2|z_2|^2}{|z_0|^2+|z_1|^2+|z_2|^2} = - \frac 12
   \quad \iff \quad |z_0|^2 = |z_2|^2
\]
holds at all points in $\mathcal{L}_2$.
Moreover this image is connected and contains the endpoints:
\[
   \mu_2 [1:0:1] = -\frac 14 \quad \mbox{ when } \quad (\alpha,\beta)=(1,0)
\]
and
\[
   \mu_2 [0:\sqrt{2}i:0] = 0 \quad \mbox{ when } \quad
   (\alpha,\beta) = \textstyle{(\frac{1}{\sqrt{2}},\frac{i}{\sqrt{2}})} \ ,
\]
hence it follows that the set $\mu (\mathcal{L}_2)$ is the full
skewed line segment in boldface depicted below.

\begin{center}
\includegraphics[scale=0.7]{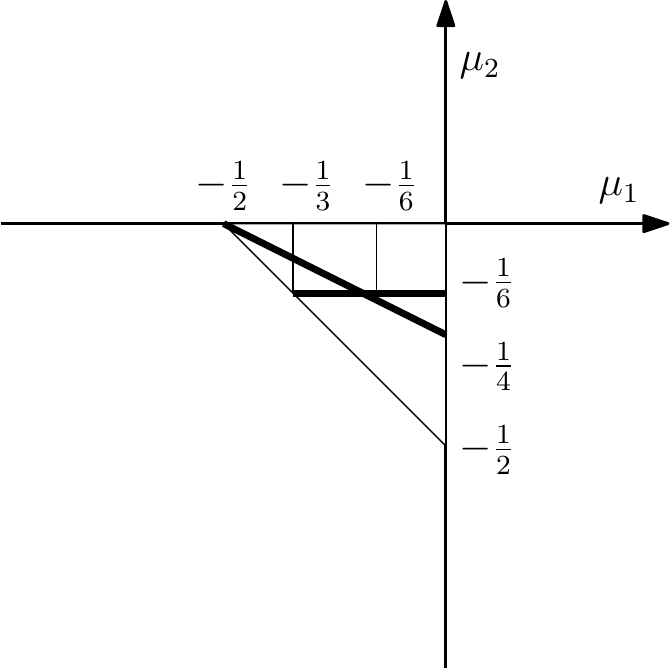}
\end{center}

Now $\mathcal{L}_2$ is invariant
under the following tilted circle action where $e^{i\theta}$ acts by
\[
   [z_0:z_1:z_2] \longmapsto [z_0:e^{i\theta}z_1:e^{2i\theta}z_2] \ .
\]
Indeed, regarding $\theta = 2 \nu$, when we replace
$(\alpha,\beta)$ by $(e^{i\nu}\alpha,e^{i\nu}\beta)$
the corresponding points of $\mathcal{L}_2$ change by
\[
   [z_0:z_1:z_2] \longmapsto [e^{-2i\nu}z_0:z_1:e^{2i\nu}z_2] =
   [z_0:e^{2i\nu}z_1:e^{4i\nu}z_2] \ .
\]

We now use action-angle coordinates $(\theta_1,\mu_1,\theta_2,\mu_2)$
which are valid {\em in the interior of the moment polytope}.
W.r.t.\ these coordinates the Fubini-Study form is
$d\theta_1 \wedge d\mu_1 + d\theta_2 \wedge d\mu_2$.
In these terms the discussion above shows that
the tangent space at any point $[z_0:z_1:z_2]$ of $\mathcal{L}_2$
with all $z_0,z_1,z_2 \neq 0$ contains the tangent vector
$\frac{\partial}{\partial \theta_1}+2\frac{\partial}{\partial \theta_2}$
and is contained in the coisotropic subspace space spanned by
$\frac{\partial}{\partial \theta_1}$,
$\frac{\partial}{\partial \theta_2}$,
$2\frac{\partial}{\partial \mu_1}-\frac{\partial}{\partial \mu_2}$
(this subspace is the kernel of $d\mu_1+2d\mu_2=0$).

So a generating system of $T\mathcal{L}_2$ may be chosen to be
$v_1 := \frac{\partial}{\partial \theta_1}+2\frac{\partial}{\partial \theta_2}$
and $v_2$ of the form
\[
   f \frac{\partial}{\partial \theta_1}
   + g \frac{\partial}{\partial \theta_2}
   + \left(2\frac{\partial}{\partial \mu_1}-\frac{\partial}{\partial \mu_2} \right)
\]
for some real functions $f$ and $g$
(we used the transversality fact ``$d\mu_2 \neq 0$ on $\mathcal{L}_2$ in the
interior of the moment polytope''; see the proof of
Lemma~\ref{lemma_tranversality}).

%We want to describe completely the tangent bundle of $\mathcal{L}_2$
%in the interior of the moment polytope.
%This will show quite clearly $\mathcal{L}_2$ intersects each
%action torus with a slope 2.

%Let $\frac{\partial}{\partial \theta_1}+2\frac{\partial}{\partial \theta_2}$
%be the first vector field generating $T\mathcal{L}_2$.
%A second linearly independent generating vector field
%must be a linear combination of the form
%\[
%   A \frac{\partial}{\partial \theta_1}
%   + B \frac{\partial}{\partial \theta_2}
%   + C \left(2\frac{\partial}{\partial \mu_1}-\frac{\partial}{\partial \mu_2} \%right) \ .
%\]
%The transversality fact that $d\mu_2 \neq 0$ on $\mathcal{L}_2$
%(Lemma~\ref{lemma_tranversality}), shows that
%\[
%   C \neq 0 \ .
%\]
%The property of $T\mathcal{L}_2$ being isotropic shows that

\begin{remark}
By shuffling the coordinates there are other two similar
lagrangian submanifolds whose images under the toric moment map
are the other two medians (line segments joining a vertex to the
midpoint of the opposing side).\\
If we consider lagrangians in $\mathbb{CP}^2$ whose image
by the toric moment map is the intersection of the moment polytope
with an affine space, there are also the trivial examples given by:
\begin{itemize}
\item
interior points (the lagrangians are the action tori),
\item
line segments {\em not through the vertices} and with specific
rational slopes (the lagrangians are the preimages by single-circle
reduction of an $\mathbb{RP}^1$ in $\mathbb{CP}^1$, which topologically
are tori) and
\item
the full polytope (the lagrangians are the standard $\mathbb{RP}^2$
and its toric rotations).
\end{itemize}
%Perhaps lagrangians of all these types in a symplectic toric manifold
%could be called {\em toric affine lagrangians}, those containing
%fixed points being ...
\end{remark}

\vspace*{1ex}

\begin{lemma}
\label{once}
In the interior of the moment polytope,
the manifold $\mathcal{L}_2$ intersects at most once each orbit of the second
circle action in $\mathbb{CP}^2$.
\end{lemma}

\begin{proof}
We compare two points in $\mathcal{L}_2$ which map to the interior of the moment
polytope, say
\[
   P = \left[ \bar{\alpha}^2 + \bar{\beta}^2 :
   \sqrt{2}(\bar{\alpha} \beta - \alpha \bar{\beta}) :
   \alpha^2 + \beta^2 \right]
\]
and
\[
   p = \left[\bar{a}^2 + \bar{b}^2 :
   \sqrt{2}(\bar{a} b - a \bar{b}) :
   a^2 + b^2 \right]
\]
where $\alpha, \beta, a, b \in \mathbb{C}$ are such that
$|\alpha|^2+|\beta|^2=1$, $|a|^2+|b|^2=1$ and all homogeneous coordinates
of $P$ and $p$ are nonzero.

The goal is to check that such points can never be nontrivially related
by the second circle action, that is,
\[
   \left[\bar{a}^2 + \bar{b}^2 :
   \sqrt{2}(\bar{a} b - a \bar{b}) :
   a^2 + b^2 \right] = \left[\bar{\alpha}^2 + \bar{\beta}^2 :
   \sqrt{2}(\bar{\alpha} \beta - \alpha \bar{\beta}) :
   e^{i\theta}(\alpha^2 + \beta^2) \right]
   \qquad \stackrel{?}{\Longrightarrow} \qquad e^{i\theta}=1
\]
% UNDERBRACKET
or equivalently
\[
   \left[ 1:
   \textstyle{
   \frac{\sqrt{2}( \bar{a} b - a \bar{b})}{\bar{a}^2 + \bar{b}^2} :
   \frac{a^2 + b^2}{\bar{a}^2 + \bar{b}^2}} \right]
   = \left[1 :
   \textstyle{
   \frac{\sqrt{2}( \bar{\alpha} \beta - \alpha \bar{\beta})}
   {\bar{\alpha}^2 + \bar{\beta}^2} :
   e^{i\theta} \cdot \frac{\alpha^2 + \beta^2}
   {\bar{\alpha}^2 + \bar{\beta}^2}} \right]
   \qquad \stackrel{?}{\Longrightarrow} \qquad e^{i\theta}=1 \ .
\]
We equate the coordinates noting that
$\bar{a} b - a \bar{b} = 2 i \Im (\bar{a}b)$ is an
imaginary number:
\[
   \begin{cases}
   \frac{2\sqrt{2} i \Im (\bar{a}b)}{\bar{a}^2 + \bar{b}^2}
   & = \quad \frac{2\sqrt{2} i \Im (\bar{\alpha}\beta)}
   {\bar{\alpha}^2 + \bar{\beta}^2} \\
   \frac{a^2 + b^2}{\bar{a}^2 + \bar{b}^2}
   & = \quad e^{i\theta} \cdot \frac{\alpha^2 + \beta^2}
   {\bar{\alpha}^2 + \bar{\beta}^2} \ .
   \end{cases}
\]
Equivalently, using converse cross-ratios, we have
\[
   \begin{cases}
   \frac{\Im (\bar{a}b)}{\Im (\bar{\alpha}\beta)}
   & = \quad \frac{\bar{a}^2 + \bar{b}^2}
   {\bar{\alpha}^2 + \bar{\beta}^2} \\
   \frac{a^2 + b^2}{\alpha^2 + \beta^2}
   & = \quad e^{i\theta} \cdot \frac{\bar{a}^2 + \bar{b}^2}
   {\bar{\alpha}^2 + \bar{\beta}^2} \ .
   \end{cases}
\]
The first equation implies that $\rho := \frac{\bar{a}^2 + \bar{b}^2}
{\bar{\alpha}^2 + \bar{\beta}^2}$ is a real number.
Since the fraction on the right-hand side of the second equation
is $\rho$ and on the left is $\bar{\rho}$,
these real numbers must be identical and hence $e^{i\theta}=1$.
\end{proof}

\begin{remark}
A different proof of Lemma~\ref{once} was first found by
Radivoje Bankovic in~\cite{radivojes_thesis}.
There he analysed $\mathcal{L}_2$ as a singular fibration over
the moment map image tilted segment, within the 2-torus fibers,
using our earlier lemmas.\\
The result in Lemma~\ref{once} is interesting in itself
since it shows a contrast between the lagrangian $\mathcal{L}_2$
and the standard embedding
of $\mathbb{RP}^2$ in $\mathbb{CP}^2$ as the fixed locus of complex
conjugation: the standard embedding intersects each orbit of
the second circle action twice (in the interior of the moment polytope),
namely for each point $[y_0:y_1:y_2]$ with $y_0,y_1,y_2 \in \mathbb{R}$
there is also the point $[y_0:y_1:-y_2]$ related by the second circle action.
Similarly for the first circle action.
The lagrangian $\mathcal{L}_2$ has also double intersections
with orbits of the first circle action:
$[ X-iY : \sqrt{2}iZ : X+iY ]$ and $[ X-iY : -\sqrt{2}iZ : X+iY ]$.
\end{remark}

\vspace*{1ex}

\begin{corollary}
\label{upshot}
The lagrangian submanifold $\mathcal{L}_2$ is a one-to-one transverse lifting
of a great circle $\ell_2$ in $\mathbb{CP}^1$.
\end{corollary}

\begin{proof}
The previous two lemmas show that $\mathcal{L}_2$ is a one-to-one
transverse lifting of a compact lagrangian submanifold $\ell_2$
in $\mathbb{CP}^1$.
In order to see that $\ell_2$ is a great circle it is enough to
note that it lies in the middle level set of $\mu_{red}$.
\end{proof}

\begin{remark}
\label{remark_monotonicity}
%Monotonicity is a key technical condition for defining Floer
%(co)homology nowadays.
By Corollary~1.2.11 of Biran and Cornea~\cite{biran-cornea}
the lagrangian $\mathcal{L}_2$ is monotone,
has minimal Maslov index $3$ and has
$HF_k(\mathcal{L}_2,\mathcal{L}_2) \simeq \mathbb{Z}_2$
for every $k \in \mathbb{Z}$.
These properties all follow directly from the fact
%\footnote{A hopefully
%explicit hamiltonian isotopy from $\mathcal{L}_2$ to the
%standard $\mathbb{RP}^2$ is in the works.}
that $\mathcal{L}_2$ is hamiltonian
isotopic to the standard $\mathbb{RP}^2$ according to Theorem~6.9
of Li and Wu~\cite{li-wu}.
%Note that, though Li and Wu credit Richard Hind to have a first
%proof of this in a preprint, that proof did not appear in the
%corresponding actually published article.\\
Having chosen the level $a=-\frac 16$ we obtain that the
corresponding lagrangian $\ell_2$ in the reduced space
is a great circle, hence also monotone.
\end{remark}

%\begin{remark}
%Another interesting question in the context of minimal lagrangian
%surfaces is: what is the actual {\em area} of $\mathcal{L}_2$?
%\end{remark}

%%%%%%%%%%%%%%%%%%%%%%%%%%%%%%%%%%%%%%%%%%%%%%%%%%%%%%%%%%%%%%%%%%%%%%%%%
%%%%%%%%%%%%%%%%%%%%%%%%%%%%%%%%%%%%%%%%%%%%%%%%%%%%%%%%%%%%%%%%%%%%%%%%%

%%%%%%%%%%%%%%%%%%%%%%%%%%%%%%%%%%%%%%%%%%%%%%%%%%%%%%%%%%%%%%%%%%%%%%%%%

\end{document}